\documentclass[11pt]{article} 

\usepackage[utf8]{inputenc} 

\usepackage{geometry} 
\usepackage{graphicx} 
\usepackage[font=sf]{caption,subcaption}
\usepackage{hyperref}
\hypersetup{
    colorlinks=false,
    linkcolor=blue,
    filecolor=blue,      
    urlcolor=blue,
}

\usepackage{array} 
\usepackage{paralist} 
\usepackage{verbatim} 
\usepackage{lipsum}
\usepackage{amsmath, amsfonts, amssymb, amsxtra, amsthm, mathrsfs}
\usepackage{enumitem}
\usepackage{color}
\usepackage{tikz,tikz-3dplot,pgfplots}

\usepackage[calcwidth]{titlesec}
\titlespacing*{\paragraph}{0pt}{0pt}{1em}
\newcommand{\periodafter}[1]{#1.}
\titleformat{\paragraph}[runin]{\bfseries}{\theparagraph}{}{\periodafter}

\newtheoremstyle{mystyle}{}{}{\itshape}{}{\bfseries}{.}{ }{\thmname{#1}\thmnumber{ #2}\thmnote{ \bfseries #3}}
\theoremstyle{plain}
\newtheorem{thm}{Theorem}[section]

\newtheorem{prop}[thm]{Proposition}
\newtheorem{lm}[thm]{Lemma}
\newtheorem{cor}[thm]{Corollary}

\theoremstyle{definition}

\DeclareMathOperator{\isom}{Isom}

\DeclareMathOperator{\hull}{hull}
\DeclareMathOperator{\core}{core}
\DeclareMathOperator{\diam}{diam}

\DeclareMathOperator{\psl}{PSL}

\DeclareMathOperator{\Stab}{Stab}
\DeclareMathOperator{\Int}{Int}
\definecolor{cof}{RGB}{219,144,71}
\definecolor{pur}{RGB}{186,146,162}
\definecolor{greeo}{RGB}{91,173,69}
\definecolor{greet}{RGB}{52,111,72}
\usetikzlibrary{calc}

\pgfplotsset{compat=1.14}

\numberwithin{equation}{section}

\title{\bfseries Existence of an exotic plane in an acylindrical 3-manifold}
\author{Yongquan Zhang}
\date{July 29, 2021}

\begin{document}
\maketitle

\begin{abstract}
Let $P$ be a geodesic plane in a convex cocompact, acylindrical hyperbolic 3-manifold $M$. Assume that $P^*=M^*\cap P$ is nonempty, where $M^*$ is the interior of the convex core of $M$. Does this condition imply that $P$ is either closed or dense in $M$? A positive answer would furnish an analogue of Ratner's theorem in the infinite volume setting.

In \cite{MMO2} it is shown that $P^*$ is either closed or dense in $M^*$. Moreover, there are at most countably many planes with $P^*$ closed, and in all previously known examples, $P$ was also closed in $M$.

In this note we show more exotic behavior can occur: namely, we give an explicit example of a pair $(M,P)$ such that $P^*$ is closed in $M^*$ but $P$ is not closed in $M$. In particular, the answer to the question above is no. Thus Ratner's theorem fails to generalize to planes in acylindrical 3-manifolds, without additional restrictions.
\end{abstract}

\section{Introduction}
This paper is a contribution to the study of topological behavior of geodesic planes in hyperbolic 3-manifolds of infinite volume.

\paragraph{Geodesic planes in hyperbolic 3-manifolds}
Let $M\cong\Gamma\backslash\mathbb{H}^3$ be an oriented, complete hyperbolic 3-manifold, presented as the quotient of hyperbolic space by a Kleinian group
$$\Gamma\subset\isom^+(\mathbb{H}^3)\cong\psl(2,\mathbb{C}).$$
Let $\Lambda\subset S^2$ be the limit set of $\Gamma$, and $\Omega=S^2-\Lambda$ the domain of discontinuity. The \emph{convex core} of $M$ is defined as
$$\core(M):=\Gamma\backslash\hull(\Lambda);$$
Equivalently, it is the smallest closed convex subset of $M$ containing all closed geodesics. Let $M^*$ be the interior of $\core(M)$. We say $M$ is \emph{convex cocompact} if $\overline{M}:=\Gamma\backslash(\mathbb{H}^3\cup\Omega)$ is compact, or equivalently $\core(M)$ is compact.

A \emph{geodesic plane} in $M$ is a totally geodesic isometric immersion $f:\mathbb{H}^2\to M$. We often identify $f$ with its image $P:=f(\mathbb{H}^2)$ and call the latter a geodesic plane as well. Given a geodesic plane $P$, write $P^*=M^*\cap P$.


\paragraph{Planes in acylindrical manifolds}
In this paper, we study the topological behavior of geodesic planes in a convex cocompact, acylindrical hyperbolic 3-manifold $M$. The topological condition of being acylindrical means that the compact Kleinian manifold $\overline{M}$ has incompressible boundary and every essential cylinder in $\overline{M}$ is boundary parallel \cite{hyperbolization1}. When $M$ has infinite volume, the property of being acylindrical is visible on the sphere at infinity: $M$ is acylindrical if and only if $\Lambda$ is a Sierpie\'nski curve\footnote{A \emph{Sierpi\'nski curve} is a compact subset $\Lambda$ of the $2$-sphere $S^2$ such that $S^2-\Lambda=\cup_{i}D_i$ is a dense union of Jordan disks with $\diam(D_i)\to 0$ and $\overline{D_i}\cap\overline{D_j}=\emptyset$ for all $i\neq j$. See Figure~\ref{fig: acy limit set} for an example.}.

When $M$ has finite volume, a geodesic plane $P$ in $M$ is either closed or dense \cite{shah, ratner}. In the infinite volume case, if we assume furthermore $\core(M)$ has totally geodesic boundary, it is shown in \cite{MMO1} that any geodesic plane $P$ in $M$ is either closed, dense in $M$, or dense in an end of $M$. In other words, geodesic planes in such an $M$ do satisfy strong rigidity properties. In \cite{MMO2}, this is generalized to all convex cocompact acylindrical 3-manifolds if we restrict to the interior of the convex core $M^*$:
\begin{thm}[\cite{MMO2}]\label{thm: rigidity}
Let $M$ be a convex cocompact, acylindrical, hyperbolic 3-manifold. Then any geodesic plane $P$ intersecting $M^*$ is either closed or dense in $M^*$.
\end{thm}
As a matter of fact, when $P^*$ is dense in $M^*$, a stronger statement holds: $P$ is actually dense in $M$. As a complement, it is also shown in \cite{MMO2} that there are only countably many geodesic planes $P$ so that $P^*$ is nonempty and closed in $M^*$. It is natural to ask, \emph{\`a la} Ratner, if these countably many planes are well-behaved topologically in the \emph{whole} manifold. For example, we have the following question in \cite{MMO2}: if $P^*$ is closed in $M^*$, is $P$ always closed in $M$? Our main theorem answers this question:
\begin{thm}\label{thm: main}
There exists a convex cocompact, acylindrical, hyperbolic 3-manifold $M=\Gamma\backslash\mathbb{H}^3$ and a geodesic plane $P$ in $M$ so that $P^*$ is nonempty and closed in $M^*$ but $P$ is not closed in $M$.
\end{thm}
Therefore, for this concrete acylindrical manifold $M$, at least one of the closed geodesic planes in $M^*$ is \emph{not} well-behaved topologically in the whole manifold. This supports the idea that the proper setting for generalizations of Ratner's theorem may be $M^*$ rather than $M$, as suggested in \cite{MMO2}.

\paragraph{Exotic planes and circles}
For simplicity, we call $P$ an \emph{exotic plane} of $M$ if $P^*$ is nonempty and closed in $M^*$ but $P$ is not closed in $M$. We now proceed to describe and visualize the example in Theorem~\ref{thm: main} from the perspective of the sphere at infinity.

Fix a presentation $M\cong\Gamma\backslash\mathbb{H}^3$. Any circle $C$ on the sphere at infinity $S^2$ determines a unique geodesic plane in $\mathbb{H}^3$, and in turn gives a geodesic plane $P$ in $M$. Conversely, given a geodesic plane $P$ in $M$, let $\tilde P$ be any lift to $\mathbb{H}^3$; its boundary at infinity is a circle $C\subset S^2$. We call $C$ a \emph{boundary circle} of $P$. Note that $\Gamma\cdot C$ gives all the boundary circles of $P$. An \emph{exotic circle} of $\Gamma$ is a boundary circle of an exotic plane in $M$.

\begin{figure}[htp]
\centering
\captionsetup{width=.8\linewidth}
\begin{minipage}{0.48\linewidth}
\centering
\begin{subfigure}[ht]{\linewidth}
\includegraphics[trim={4.5cm 5cm 4.5cm 5cm},clip,width=\textwidth]{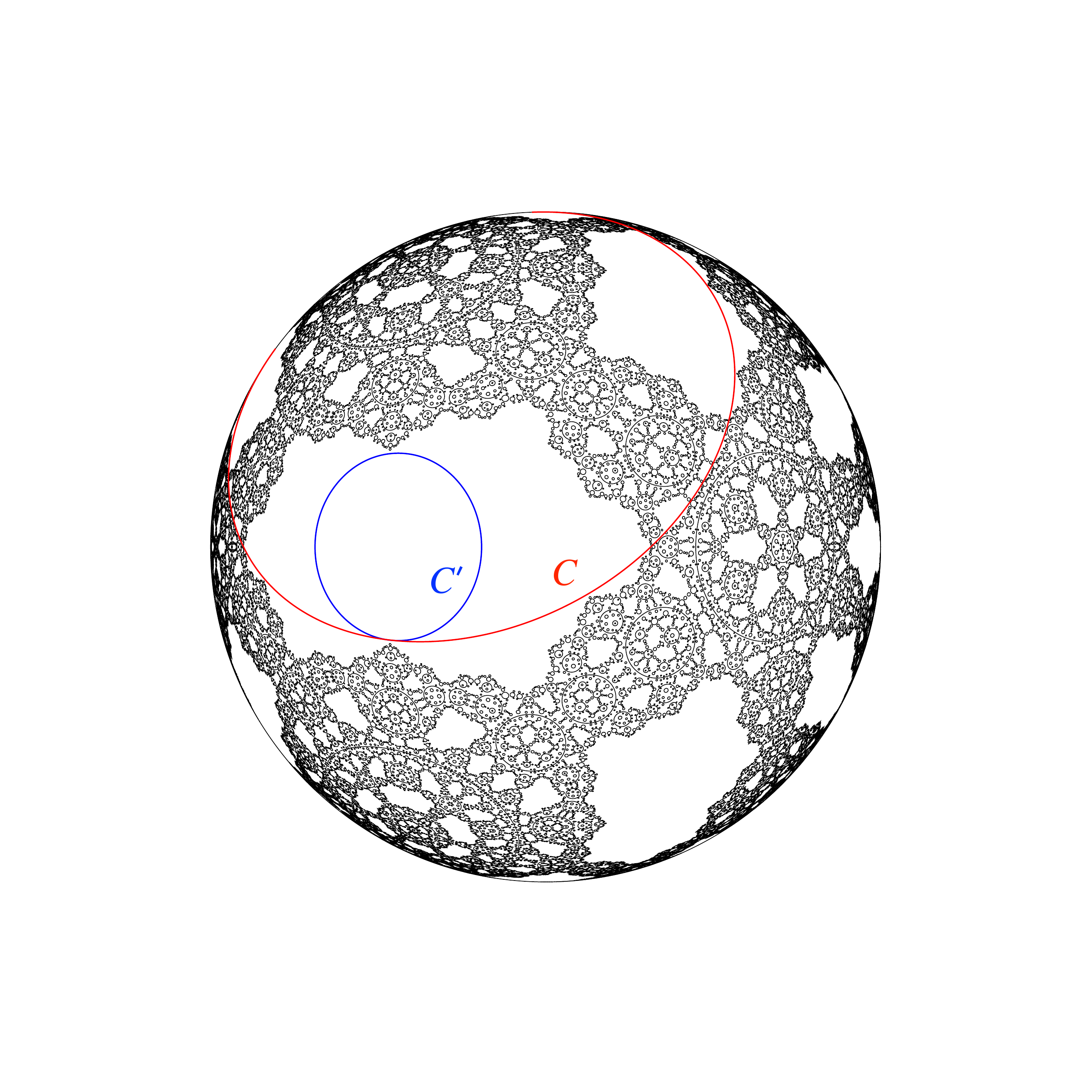}
\caption{The limit set $\Lambda$ and an exotic circle $C$}
\label{fig: acy limit set}
\end{subfigure}
\end{minipage}
\begin{minipage}{0.48\linewidth}
\centering
\begin{subfigure}[ht]{\linewidth}
\includegraphics[trim={4.5cm 5cm 4.5cm 5cm},clip,width=\textwidth]{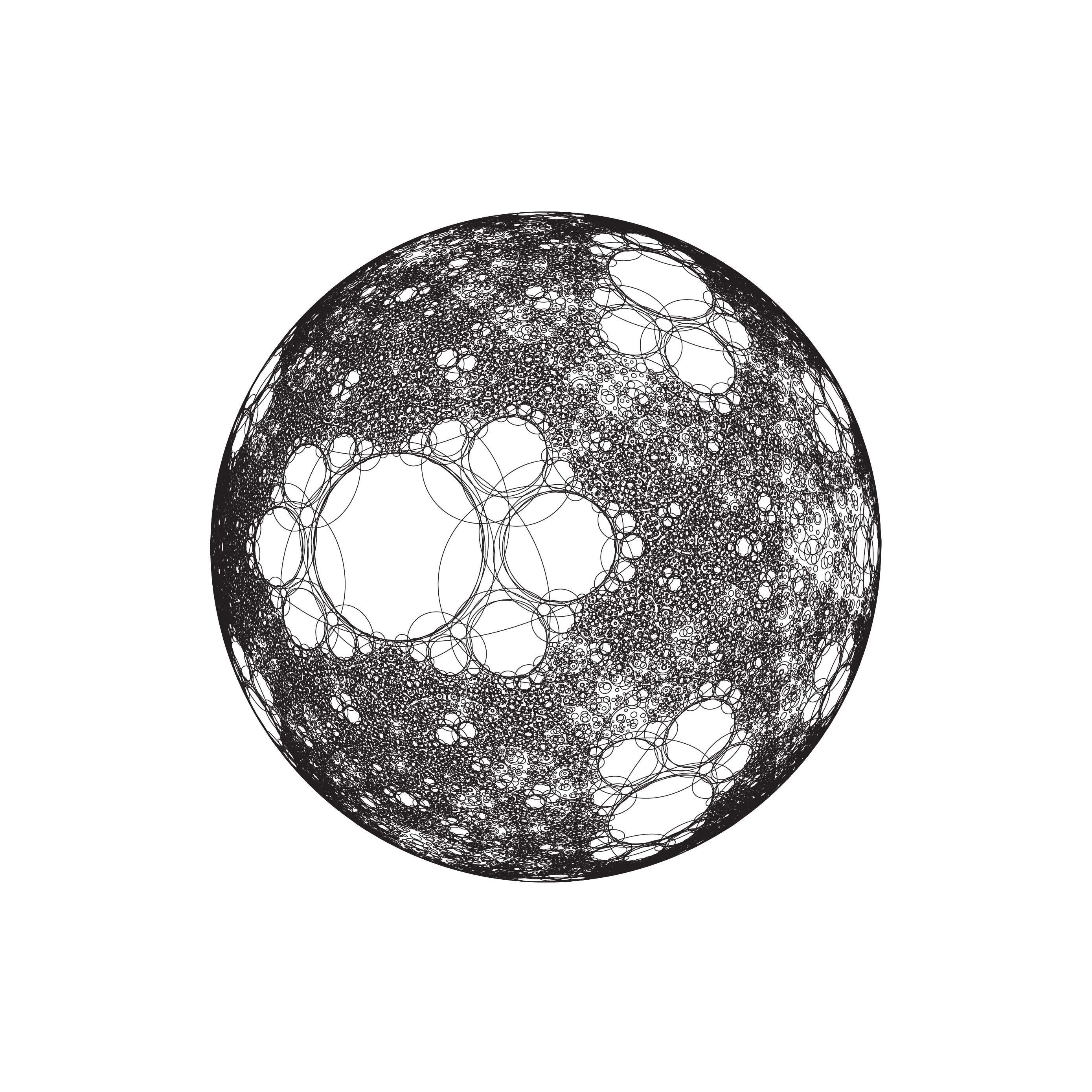}
\caption{The orbit $\Gamma\cdot C$}
\label{fig: acy exotic plane}
\end{subfigure}
\end{minipage}
\caption{The example in Theorem~\ref{thm: main} from the perspective of the sphere at infinity. Note that the orbit $\Gamma\cdot C$ of the exotic circle $C$ limits on $C'\notin\Gamma\cdot C$.}
\label{fig: acylindrical}
\end{figure}

Figure~\ref{fig: acylindrical} gives some visualizations of the example in Theorem~\ref{thm: main}: Figure~\ref{fig: acy limit set} depicts the limit set of $\Gamma$ with an exotic circle $C$ marked, and Figure~\ref{fig: acy exotic plane} shows the orbit of $C$ under $\Gamma$. Note that $C'$ (also marked in Figure~\ref{fig: acy limit set}) is not a circle in $\Gamma\cdot C$, but there exists a sequence $\gamma_i\in \Gamma$ so that $\gamma_iC\to C'$. This is reflected in our discussion of the geometry of $P$ (see below), and is quite visible from Figure~\ref{fig: acy exotic plane}.

\paragraph{Geometry of the exotic plane}
The exotic plane in Theorem~\ref{thm: main} comes from a particular orbifold $\mathcal{O}=\mathcal{O}(t_0)$ in the one-parameter family of acylindrical orbifolds $\mathcal{O}(t)$ constructed in \cite{double_lunchbox}. We will actually construct an exotic plane $P$ in the orbifold $\mathcal{O}$; this orbifold has a finite manifold cover $M$ by Selberg's lemma, and any lift of $P$ to $M$ is then exotic. For simplicity, we will mostly focus on describing the exotic plane in $\mathcal{O}$; properties of its lift to the manifold cover then follow.

Note that an exotic plane \emph{must} accumulate on the convex core boundary. In our example, the plane $P$ is a nonelementary, convex cocompact surface with one infinite end; its restriction to $\core(\mathcal{O})$ cuts the infinite end into a crown with two tips. In particular $P^*=P\cap\Int(\core(\mathcal{O}))$ has finite area. The two tips of the crown wrap around and tends to the bending geodesic on the boundary of $\core(\mathcal{O})$. Finally, the closure $\overline{P}=P\cup P'$, where $P'$ is a closed geodesic plane contained in the infinite end of $M$. As a matter of fact, $P'$ is a cylinder whose core curve is precisely the bending geodesic. Note that $\overline{P}$ is not a closed suborbifold of $\mathcal{O}$, as it is not locally connected near $P'$.

Figure~\ref{fig: geometry} is a picture of $P$ near the convex core boundary in the quotient of $\mathcal{O}$ by a reflection symmetry. We remark that the behavior of $P$ near the convex core boundary only depends on the geometry of the corresponding quasifuchsian orbifold; see Theorems~\ref{thm: exotic circle} and \ref{thm: exotic}, and Figure~\ref{fig: prop proof} for details. 

\begin{figure}[htp]
\centering
\captionsetup{width=.6\linewidth}
\includegraphics{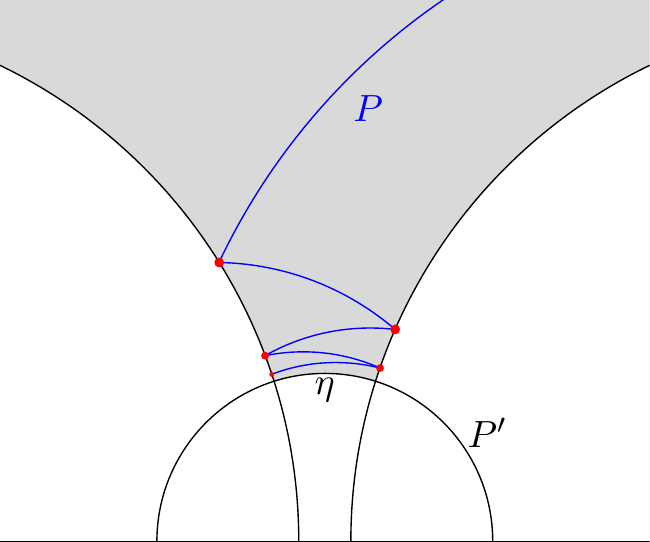}
\caption{A cross section near the convex core boundary, in a plane orthogonal to $P$ and containing $\eta$, the bending geodesic.}
\label{fig: geometry}
\end{figure}

\paragraph{Planes in cylindrical manifolds}
We remark that for \emph{cylindrical} manifolds, there is no analogue of Theorem~\ref{thm: rigidity} in general. For example, geodesic planes intersecting the convex core of a quasifuchsian manifold may have non-manifold, even fractal closures, as explained in \cite[Appx.~A]{MMO1}. Nevertheless, we can still consider exotic planes in a general convex cocompact hyperbolic 3-manifold $M$. In fact, we give examples of exotic planes in a concrete family of quasifuchsian manifolds in \S\ref{sec: quasifuchsian}.

In a general convex cocompact manifold $M$ with incompressible boundary, there may be uncountably many geodesic planes closed in $M^*$. For example, in \cite[\S2]{MMO2}, the authors described a continuous family of hyperbolic cylinders closed in $M^*$ for a quasifuchsian manifold $M$. Nevertheless, we show in the appendix:
\begin{thm}\label{thm: countable}
Let $M$ be a convex cocompact hyperbolic 3-manifold with incompressible boundary. Then there are at most countably many exotic planes in $M$.
\end{thm}
In particular, the union of all exotic planes in $M^*$ has measure zero.

\paragraph{Exotic horocycles}
A \emph{horocycle} in a complete hyperbolic 3-manifold $M$ is an isometrically immersed copy of $\mathbb{R}$ with zero torsion and geodesic curvature $1$. Here zero torsion means that it is contained in a geodesic plane. When $M$ has finite volume, the closure of any horocycle is a properly immersed submanifold, also a consequence of the general Ratner's theorem \cite{ratner}. This is extended to convex cocompact and acylindrical $M$ whose convex core has totally geodesic boundary \cite{acy_horo}, using the complete classification of geodesic planes in \cite{MMO1} mentioned above.

Our example in Theorem~\ref{thm: main} similarly shows that such rigidity does not hold in general for acylindrical manifolds. Indeed, the closure of a horocycle that is dense in the exotic plane is not a submanifold. Such a horocycle exists by e.g.\ \cite{horocycle}.
\paragraph{Questions}
We conclude the discussion by mentioning the following open questions.
\begin{enumerate}[topsep=3pt, itemsep=3pt]
\item The example in Theorem~\ref{thm: main} is fairly tame, but wilder examples may exist when the bending lamination is nonatomic. For example, is there an exotic plane $P$ so that $P^*$ has infinite area?
\item In both quasifuchsian and acylindrical cases, we give examples containing one exotic plane. In view of Theorem~\ref{thm: countable}, a natural question is if the result is ``sharp". That is, does there exist a convex cocompact hyperbolic 3-manifold with infinitely many exotic planes?
\end{enumerate}

\paragraph{Notes and references}
The rigidity results of \cite{MMO1, MMO2} have recently been extended to certain geometrically finite acylindrical manifolds in \cite{acy_geom_finite}.

This paper is organized as follows. Section~\ref{sec: quasifuchsian} is devoted to examples of exotic planes in a concrete family of quasifuchsian orbifolds. We then leverage the quasifuchsian examples to construct an acylindrical example, giving a proof of Theorem~\ref{thm: main} in Section~\ref{sec: construction}. Finally, we calculate explicitly the parameters for our example in Section~\ref{sec: computation}; they are needed to produce Figure~\ref{fig: acylindrical}.

\paragraph{Acknowledgements}
I would like to thank C.\ McMullen for enlightening discussions, and for sharing his example of a quasifuchsian exotic plane, which appears in Section~\ref{sec: quasifuchsian}. I also want to thank H.\ Oh for raising the question about exotic horocycles. Figure~\ref{fig: polyhedron_hyperbolic} was produced by \textbf{Geomview} \cite{geomview}, and computer scripts written by R.\ Roeder \cite{roe}. Figures~\ref{fig: acylindrical}, \ref{fig: limit set} and \ref{fig: quasifuchsian exotic plane} were produced by McMullen's program \textbf{lim} \cite{lim}. Finally, I want to thank the anonymous referee for their helpful comments and suggestions.

\section{Quasifuchsian examples}\label{sec: quasifuchsian}
In this section, we give examples of exotic planes in a concrete family of quasifuchsian orbifolds.

Fix an integer $n\ge3$. Let $R$ be a quadrilateral in the extended complex plane whose sides are either line segments or circular arcs, and whose interior angles are all $\pi/n$. Reflections in the sides of $R$ generate a discrete subgroup of $\psl(2,\mathbb{C})$, and let $\Gamma_R$ be its index 2 subgroup of orientation preserving elements. We note that $\Gamma_R$ is a quasifuchsian group, and the corresponding quasifuchsian orbifold $N_R:=\Gamma_R\backslash\mathbb{H}^3$ is homotopic to a sphere with 4 cone points of order $n$.

Denote by $C_i,1\le i\le4$ the circles on which the four sides of $R$ lie, so that $C_1$ and $C_3$ contain opposite sides. The corresponding hyperbolic planes in $\mathbb{H}^3$ are also denoted by $C_i$. To construct $N_R$, one can take two copies of the infinite ``tube" bounded by these four planes and identify corresponding faces. A fundamental domain for $\Pi_R$ is thus two copies of the tube, although for most discussions we consider the full reflection group with a fundamental domain simply being the tube.

Two copies of the geodesic segment perpendicular to both $C_1$ and $C_3$ glue up to a closed geodesic $\xi$ in $N_R$, and similarly the other two planes give a closed geodesic $\eta$. Let $2\cosh^{-1}(s)$ and $2\cosh^{-1}(t)$ be the lengths of $\xi$ and $\eta$ respectively. Then
\begin{prop}
For any quadrilateral $R$, we have $(s-1)(t-1)\le 4\cos^2(\pi/n)$, with equality if and only if $\Gamma_R$ is Fuchsian. When the inequality is strict, the convex core of $N_R$ is bent along $\xi$ on one side and $\eta$ on the other.
\end{prop}
This proposition is a combination of Proposition~4.1 and Lemma 4.3 in \cite{double_lunchbox}. See Figure~\ref{fig: quasifuchsian} for some visualizations with $n=3$.
\begin{figure}[htp]
\centering
\captionsetup{width=.8\linewidth}
\begin{subfigure}[t]{0.48\linewidth}\centering
\includegraphics[trim={3.5cm 10cm 3.5cm 10cm},clip,width=\textwidth]{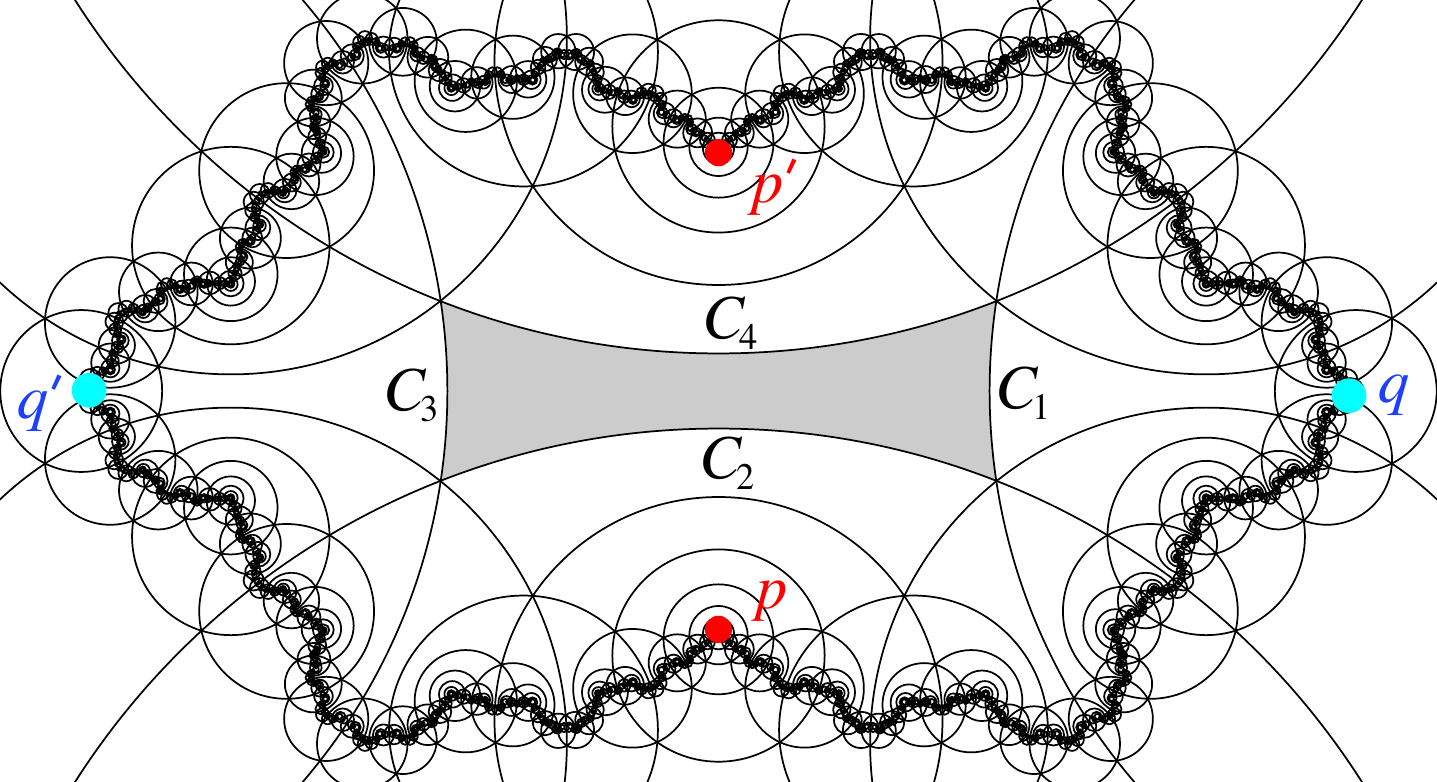}
\caption{Quadrilateral $R$ in gray and limit set of $\Gamma_R$.}
\label{fig: limit set}
\end{subfigure}
\begin{subfigure}[t]{0.48\linewidth}
\centering
\includegraphics{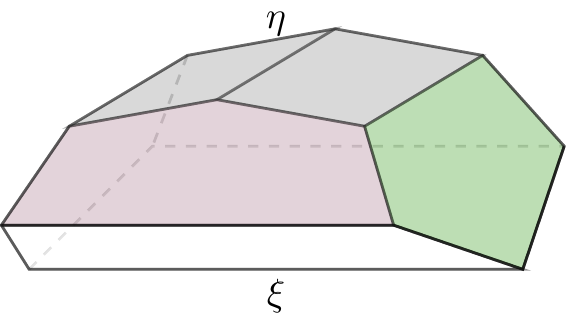}
\caption{A combinatorial picture of $\core(N_R)$}
\label{fig: quasifuchsian convex core}
\end{subfigure}
\caption{An example of the main construction of a quasi-Fuchsian group $\Gamma_R$ from a quadrilateral $R$, for which we determine the bending lamination $\eta,\xi$ facing toward and away from $R$, respectively. A lift of $\eta$ has end points $p,p'$, and a lift of $\xi$ has end points $q,q'$.}
\label{fig: quasifuchsian}
\end{figure}
Note that Figure~\ref{fig: limit set} is drawn so that $R$ is centered at the origin, and symmetric across real and imaginary axes. The end points $p,p'$ of a lift $\tilde\eta$ of $\eta$ then lie on the imaginary axis; the corresponding hyperbolic element, also denoted by $\tilde\eta$, is given by reflection across $C_2$ followed by refection across $C_4$. Similarly, the end points $q,q'$ of a lift $\tilde\xi$ of $\xi$ lie on the real axis, see Figure~\ref{fig: limit set}.

\begin{figure}[htp]
\centering
\begin{minipage}{0.48\linewidth}
\begin{subfigure}[ht]{\linewidth}
\includegraphics[trim={3.5cm 10cm 3.5cm 10cm},clip,width=\textwidth]{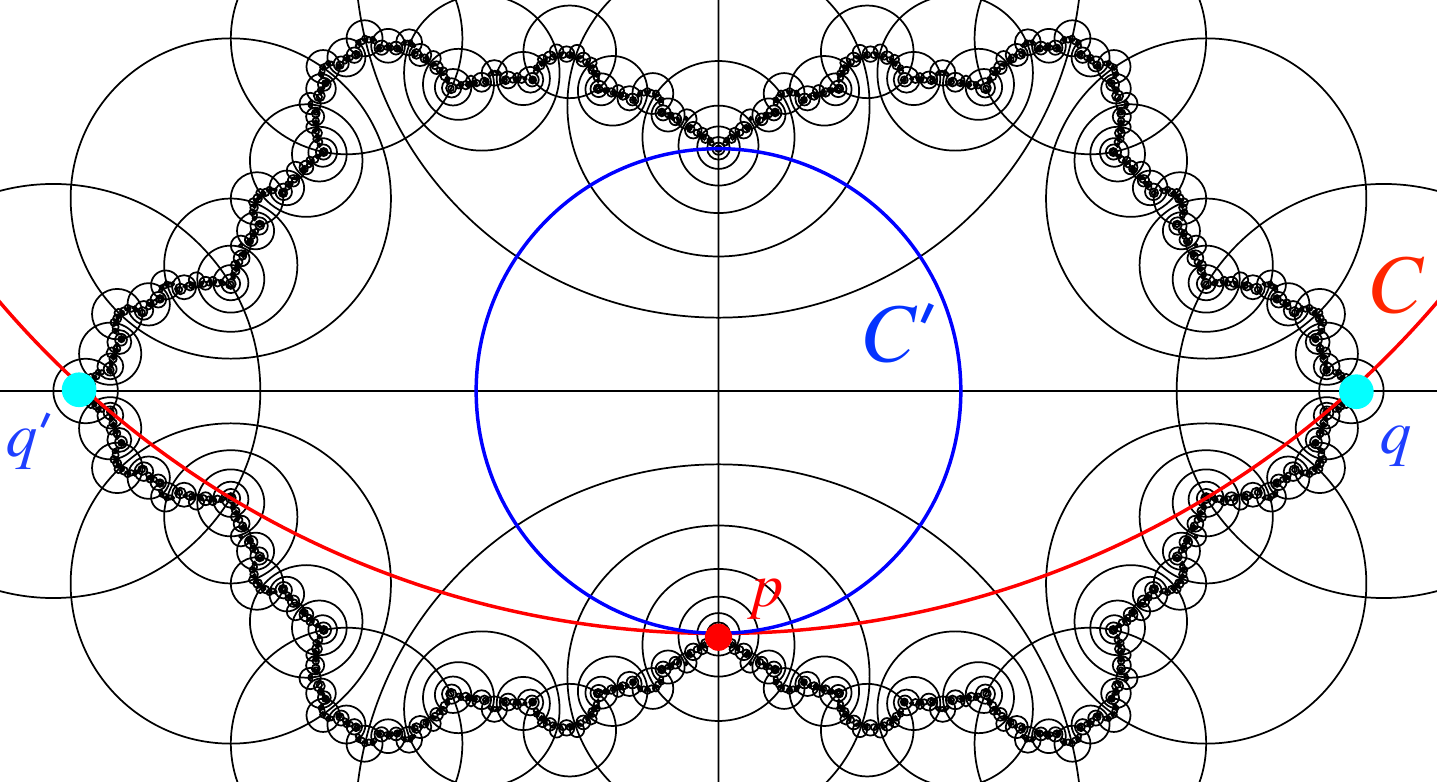}
\caption{Axes of bending geodesics and an exotic circle}
\label{fig: axes}
\end{subfigure}
\end{minipage}
\begin{minipage}{0.48\linewidth}
\centering
\begin{subfigure}[ht]{\linewidth}
\includegraphics[trim={3.5cm 10cm 3.5cm 10cm},clip,width=\textwidth]{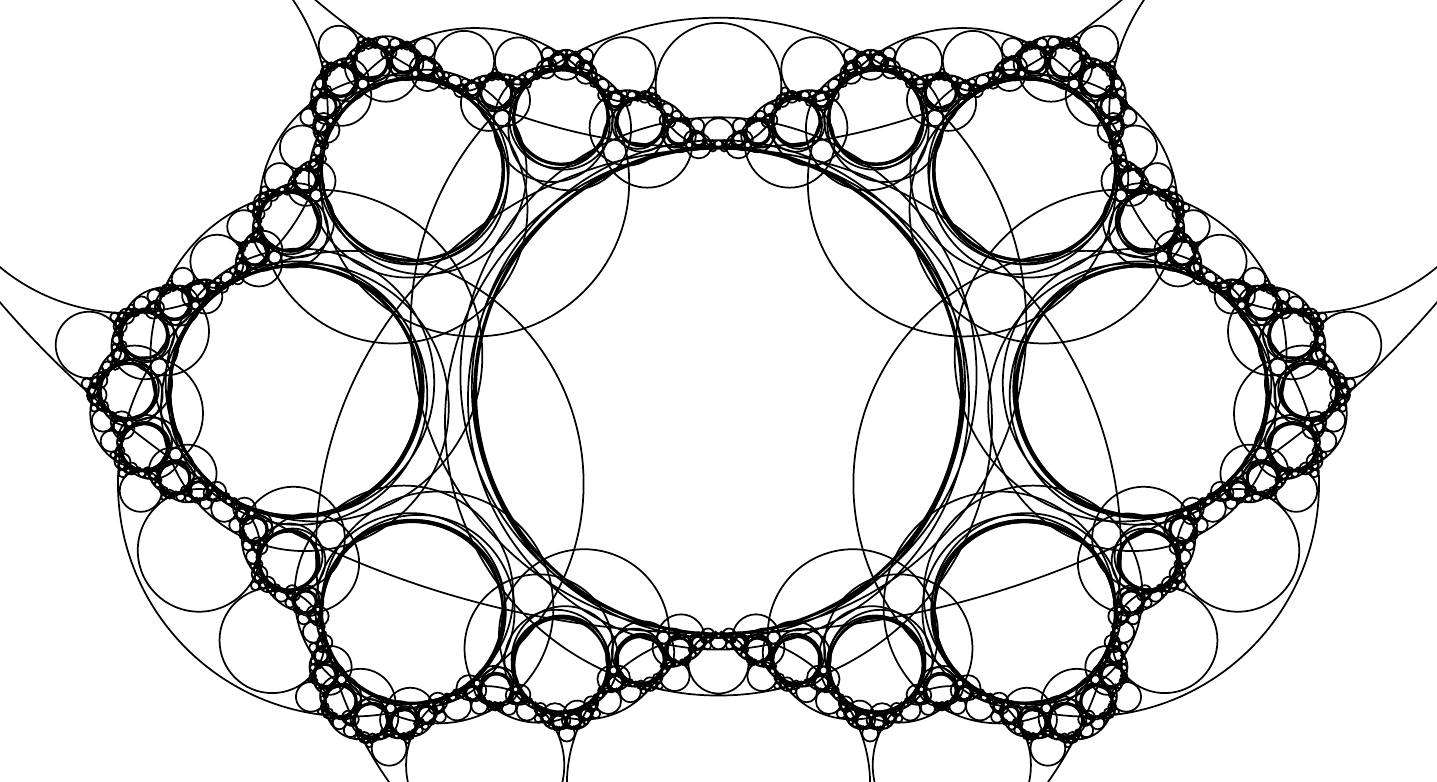}
\caption{Orbit of the exotic circle}
\label{fig: exotic plane}
\end{subfigure}
\end{minipage}
\caption{An exotic circle for quasifuchsian group $\Gamma_R$}
\label{fig: quasifuchsian exotic plane}
\end{figure}

Let $\Lambda$ be the limit set of $\Gamma_R$. Take the circle $C$ passing through the end points $q,q'$ of $\tilde\xi$ and the repelling fixed point $p$ of $\tilde\eta$; see Figure~\ref{fig: axes}. It is easy to see that $p$ is an isolated point in $C\cap\Lambda$. As $C$ is stabilized by reflections in $C_1$ and $C_3$, it also passes through the orbit of $p$ under the group generated by these reflections. It is clear that $C\cap\Lambda$ consists of points in this orbit together with the end points $q,q'$ of $\tilde\xi$. We have

\begin{thm}\label{thm: exotic circle}
For any quadrilateral $R$, the circle $C$ is exotic. That is, the corresponding geodesic plane $P$ is closed in $N_R^*$ but not in $N_R$.
\end{thm}

\begin{proof}
For simplicity, we will suppress the subscript $R$. Since $C$ passes through the repelling fixed point of $\tilde\eta$, the sequence $\tilde{\eta}^n\cdot C$ tends to a circle $C'$ passing through both end points of $\tilde\eta$; in particular, the geodesic plane $P$ determined by $C$ is not closed in $N$, and accumulates on the plane $P'$ determined by $C'$. This is clearly visible in Figure~\ref{fig: exotic plane}, where the orbit of $C$ under $\Gamma$ is drawn.

\begin{figure}[ht!]
\centering
\begin{subfigure}[ht]{\linewidth}
\captionsetup{width=.5\linewidth}
\centering
\includegraphics{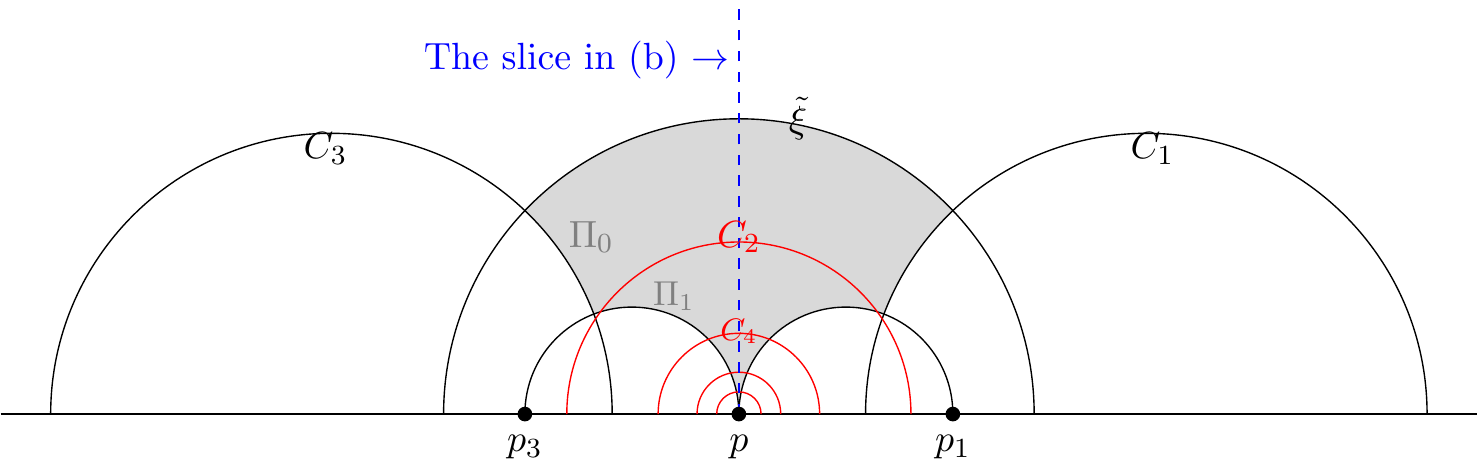}
\caption{A fundamental domain inside $\tilde P$. The red circles mark where it meets $C_2,C_4$ and their orbits}
\label{fig: fundamental domain}
\end{subfigure}
\vspace{10pt}

\begin{subfigure}[ht]{\linewidth}
\centering
\includegraphics{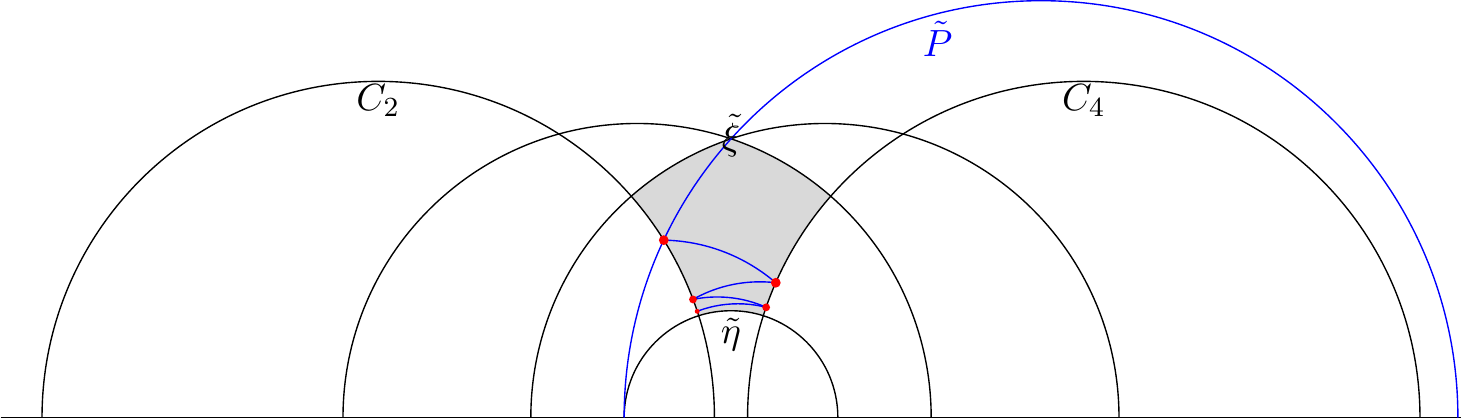}
\caption{A perpendicular slice of $\tilde P$}
\label{fig: billiard}
\end{subfigure}
\caption{Some visualizations for the proof of Thm.~\ref{thm: exotic circle}}
\label{fig: prop proof}
\end{figure}

It remains to show $P^*$ is closed in $N^*$. The hyperbolic plane $\tilde P$ in $\mathbb{H}^3$ determined by $C$ is divided into two half planes by $\tilde\xi$. One half descends to a half cylinder contained in an end of $N$; indeed, this half plane is stabilized by $\tilde\xi$ and a fundamental domain for the action of this hyperbolic element is compact in $\mathbb{H}^3\cup\Omega$, where $\Omega$ is the domain of discontinuity of $\Gamma$, and thus our assertion follows from proper discontinuity of the action. 

For the other half, again since it is stabilized by $\tilde\xi$, it suffices to consider a fundamental domain under the action of this hyperbolic element. As a matter of fact, we may even consider a fundamental domain under the full reflection group. One choice of this is the portion of the half plane sandwiched between $C_1$ and $C_3$. Let $p_1$ and $p_3$ be the images of $p$ under reflections across $C_1$ and $C_3$ respectively. The geodesic with end points $p$ and $p_1$ descends to a complete geodesic contained in the convex core boundary, and same for the geodesic with end points $p$ and $p_3$. Hence to understand $P^*$ we only need to consider the part of the fundamental domain mentioned above that lies between the geodesics $pp_1$ and $pp_3$; see the gray region in Figure~\ref{fig: fundamental domain}. We denote this part of $\tilde P$ by $\Pi$.

This portion is divided by orbits of $C_2$ and $C_4$ into countably many pieces $\Pi_0, \Pi_1,\ldots$ (see Figure~\ref{fig: fundamental domain}). Each piece is relatively compact (since its closure is disjoint from the boundary at infinity) and descends to a piece contained entirely in $\core(N)$.

Let $l$ be the intersection of $\Pi$ and the hyperbolic plane whose boundary at infinity is $i\mathbb{R}$ (in Figure~\ref{fig: fundamental domain}, this is part of the dotted blue line starting at its intersection with $\tilde\xi$ and ending at $p$). Let $l_i$ be the segment of $l$ contained in each piece $\overline{\Pi_i}$. Since $l$ and $\tilde\eta$ shares an end point, the distance\footnote{That is, the maximal distance between two points from each set.} between each segment $l_i$ and $\tilde\eta$ goes to $0$. In $N$, we conclude that the distance between the projection of $l_i$ and $\eta$ goes to zero; see Figure~\ref{fig: billiard}.

Finally, the distance between $\overline{\Pi_i}$ and $\tilde\eta$ goes to $0$. Indeed, since $\Pi_i$ is bounded by $p_1p$ and $p_3p$ (at least when $i$ is large enough), the distance of a point in $\overline{\Pi_i}$ to the segment $l_i$ goes to zero. Together with the discussion in the last paragraph, this gives the desired claim. We conclude that in $N$, the distance between the projection of each $\overline{\Pi_i}$ and $\eta$ goes to zero.

This implies that $P^*$ only accumulates on $\eta$, so it is closed in $N^*$.
\end{proof}
The proof above gives a clear picture of the topology of $P^*$. Indeed, we described a fundamental domain under the action of the full reflection group (the gray region in Figure~\ref{fig: fundamental domain}). Two pieces of this gives a crown with two tips, and $P^*$ is precisely the interior of this crown properly immersed in $N^*_R$.

It is also easy to understand the behavior of $P$ in $N_R$. Recall that $P'$ is the geodesic plane in $N_R$ determined by $C'$, as described in the proof above. We have:
\begin{thm}\label{cor: qf_geometry}
The geodesic plane $P$ only accumulates on $P'$ (i.e.\ $\overline{P}=P\cup P'$), where $P'$ is a properly immersed hyperbolic cylinder whose intersection with $\core(N_R)$ is the bending geodesic $\eta$. Moreover, $\overline{P}$ is not locally connected, and thus not a suborbifold of $N_R$.
\end{thm}
\begin{proof}
Again we will suppress the subscript $R$ for simplicity. First note that $P'$ is closed in $N$. This similarly follows from proper discontinuity: the plane $\tilde P'$ determined by $C'$ is stabilized by $\tilde\eta$, and a fundamental domain for its action is compact in $\mathbb{H}^3\cup\Omega$. Since $\tilde P'$ intersects the convex hull of the limit set $\Lambda$ at $\tilde \eta$,  we have $P'\cap\core(N)=\eta$ as desired.

To understand the topology of $P$, we again look at Figure~\ref{fig: prop proof}. Instead of the portion of $\tilde P$ lying between the geodesics $pp_1$ and $pp_3$, we need to consider the whole fundamental region $\Sigma$ sandwiched between $C_1$ and $C_3$. Similarly, it is divided by orbits of $C_2$ and $C_4$ into countably many pieces $\Sigma_0, \Sigma_1,\ldots$, and each piece is relatively compact in $\mathbb{H}^3\cup\Omega$. Thus a sequence of points in $P$ limit to a point not in $P$ only when it meets the projection of infinitely many of these pieces.

Let $\tau_2,\tau_4$ be reflections across $C_2, C_4$ respectively. It is clear from Figure~\ref{fig: billiard} that $\tau_2\Sigma_1$ is sandwiched between $C_2$ and $C_4$; similarly, $\tau_4\tau_2\Sigma_2=\tilde\eta\Sigma_2$ is sandwiched between $C_2$ and $C_4$. Inductively, $\tilde\eta^{-k}\tau_2\Sigma_{2k+1}$ and $\tilde\eta^k\Sigma_{2k}$ are sandwiched between $C_2$ and $C_4$. Since $\Sigma_i$ is contained in $\tilde P$, and both $\tilde\eta^k\tilde P$ and $\tilde\eta^{-k}\tau_2\tilde P$ converge to $\tilde P'$ as $k\to\infty$, these pieces limit on the portion of $\tilde P'$ sandwiched between $C_2$ and $C_4$.

This implies that $P$ only accumulates on $P'$. Moreover, any neighborhood of a point on $P'$ intersects $P$ in infinitely many pieces coming from each $\Sigma_i$, so $\overline{P}=P\cup P'$ is not locally connected.
\end{proof}
Note that Theorem~\ref{cor: qf_geometry} implies Theorem~\ref{thm: exotic circle}, since $P^*\subset P$ and $P'$ is disjoint from the interior of $\core(N)$. Nevertheless, we keep the proof of Theorem~\ref{thm: exotic circle} to give a thorough description of the geometry of $P^*$ as well.

\section{An acylindrical example}\label{sec: construction}
In this section, we briefly review the acylindrical orbifolds constructed in \cite{double_lunchbox}, which are covered by the quasifuchsian orbifolds described in the last section. We then construct an exotic plane in one of these acylindrical orbifolds by projecting down the corresponding quasifuchsian example.

\noindent\textbf{The example manifold and its deformations.}\quad One way to construct explicit examples of hyperbolic 3-manifolds is gluing faces of hyperbolic polyhedra (with desired properties) via hyperbolic isometries. See \cite[\S3.3]{thurston_book}, \cite{poly1}, \cite{poly2} and \cite{gaster} for some acylindrical examples.

Along the same idea, we construct a hyperbolic polyhedron given by the combinatorial data encoded in the Coxeter diagram shown in Figure~\ref{fig: coxeter}.  Let $\tilde Q$ be the hyperbolic polyhedron with an infinite end obtained by extending across Face 1 to infinity. In other words, it is the hyperbolic polyhedron with one infinite end bounded by the hyperbolic planes containing Faces~2--10; see Figure \ref{fig: polyhedron_hyperbolic}. Reflections in all faces of $\tilde Q$ generate a discrete subgroup of hyperbolic isometries; a subgroup of index $2$ gives an acylindrical hyperbolic 3-orbifold with Fuchsian end. We can deform $\tilde Q$ by pushing closer or pulling apart Faces $3$ and $5$, fixing the dihedral angles, and obtain deformations of the orbifold. This gives \cite[Thm.~1.2]{double_lunchbox}:
\begin{figure}[htp]
\centering
\begin{minipage}{0.45\linewidth}
\begin{subfigure}[t]{\linewidth}
\centering
\includegraphics{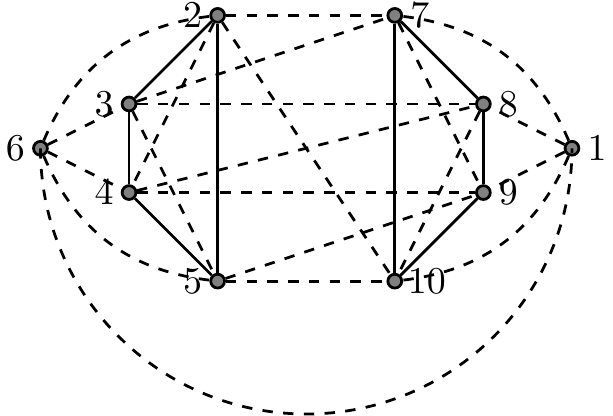}
\caption{The Coxeter diagram}
\label{fig: coxeter}
\end{subfigure}
\vspace{10pt}

\begin{subfigure}[t]{\linewidth}
\centering
\includegraphics{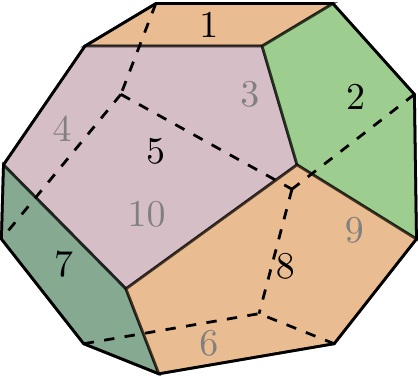}
\caption{The corresponding polyhedron}
\label{fig: polyhedron}
\end{subfigure}
\end{minipage}
\hspace{0.05\linewidth}
\begin{minipage}{0.4\linewidth}
\begin{subfigure}[t]{\linewidth}
\centering
\includegraphics[width=\textwidth]{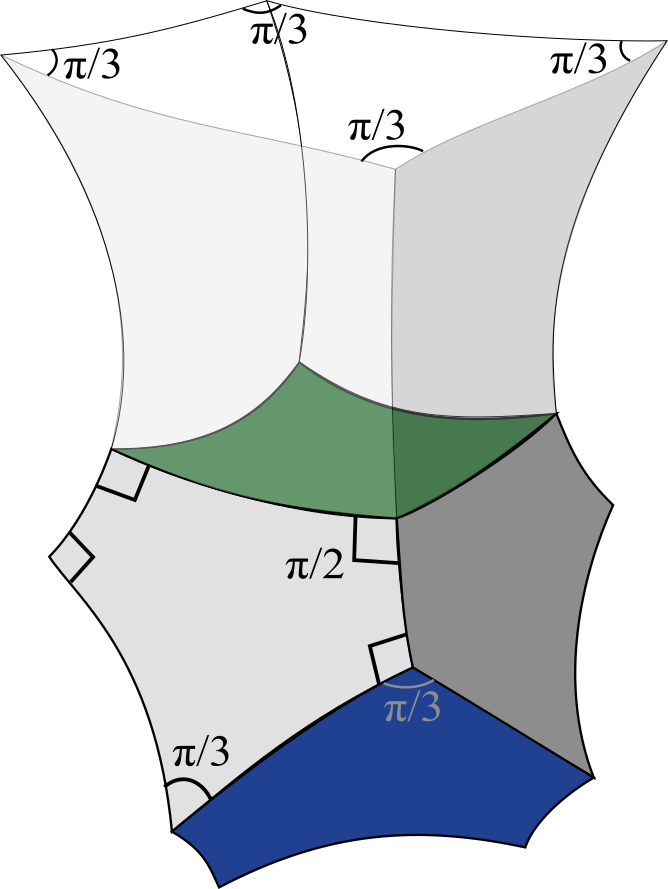}
\caption{The hyperbolic polyhedron in the unit ball model}
\label{fig: polyhedron_hyperbolic}
\end{subfigure}
\end{minipage}
\caption{Combinatorial data and visualization of the polyhedron}
\label{fig: polyhedron_data}
\end{figure}
\begin{thm}\label{thm: acy_group}
For each $t\in\left(1,\dfrac{5+\sqrt{39}}{3}\right)$, there exists a unique hyperbolic polyhedron $\tilde Q(t)$ so that the hyperbolic distance between Faces $3$ and $5$ is $\cosh^{-1}(t)$. The corresponding hyperbolic orbifold $\mathcal{O}(t)$ is acylindrical convex cocompact, whose convex core boundary is totally geodesic if and only if $t=2$.
\end{thm}

See \cite[Figure 3]{double_lunchbox} for some samples of the deformation. One way to explicitly construct $\mathcal{O}(t)$ is to take two copies of $\tilde Q(t)$ and identify corresponding faces. A fundamental domain for the corresponding Kleinian group is thus two copies of $\tilde Q(t)$, although for most discussions we often consider the full reflection group with a fundamental domain simply being $\tilde Q(t)$.

It is clear from construction that the quasifuchsian orbifold $N(t)$ corresponding to the boundary of $\mathcal{O}(t)$ is an example of those described in Section~\ref{sec: quasifuchsian}. Consistent with the notations there, let $\eta$ be the simple closed geodesic in $\mathcal{O}(t)$ coming from two copies of the geodesic segment orthogonal to both Faces 3 and 5, and $\xi$ be that from Faces 2 and 4. To distinguish the two ends of $N(t)$, we call the end shared with $\mathcal{O}(t)$ the \emph{top end}, and the other \emph{bottom end}. We have \cite[Thm.~1.3]{double_lunchbox}:

\begin{thm}\label{thm: skinning_map}
The convex core boundary of $\mathcal{O}(t)$ is bent along $\eta$ for $t\in(1,2)$, and $\xi$ for $t\in(2,(5+\sqrt{39})/{3})$, with hyperbolic lengths $l_\eta(t)=2\cosh^{-1}(t)$ and $l_\xi(t)=2\cosh^{-1}(s)$, where $s=\phi(t)$ is an explicit monotonic function. Bending angles $\lambda_\eta(t)$ and $\lambda_\xi(t)$ can also be explicitly calculated. On the bottom end of $N(t)$, the convex core boundary is bent along $\xi$ for $t\in(1,2)$ and $\eta$ for $t\in(2,(5+\sqrt{39})/{3})$, and lengths and angles can also be similarly calculated.
\end{thm}

\begin{figure}[htp]
\centering
\captionsetup{width=.6\linewidth}
\includegraphics{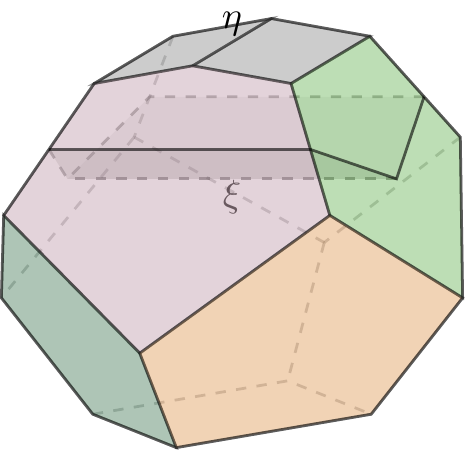}
\caption{A combinatorial picture of $\core(\mathcal{O}(t))$ for $t\in(1,2)$. Note that a copy of $\core(N(t))$ is embedded in the polyhedron}
\label{fig: poly bending}
\end{figure}

We refer to \cite{double_lunchbox} for proofs of these statements and explicit calculations.

\noindent\textbf{Existence of an exotic plane.}\quad We wish to make use of the quasifuchsian examples in Section~\ref{sec: quasifuchsian} to construct an acylindrical example. In particular, if we take the same plane $P(t)$ in $N(t)$ and project it down to $\mathcal{O}(t)$ via the covering map $N(t)\to \mathcal{O}(t)$ of infinite degree, we have to make sure that the cylinder contained in the bottom end of $N(t)$ projects to a closed surface in $\mathcal{O}(t)$. For this, we have
\begin{thm}\label{thm: exotic}
There exists $t=t_0\in(1,2)$ so that the image of $P(t_0)$ under the projection $N(t_0)\to \mathcal{O}(t_0)$ is an exotic plane. Moreover, the closure of the plane in $\mathcal{O}(t_0)$ is not locally connected, and hence not a suborbifold.
\end{thm}
\begin{proof}
For this, we look at how the plane intersects a fixed fundamental domain, the polyhedron $\tilde Q(t)$. A lift $\tilde P(t)$ of the plane $P(t)$ passes through the geodesic segment orthogonal to Faces 2 and 4, so it is orthogonal to both faces as well. When $t\to2$, $\tilde P(t)$ tends to Face 1; when $t\to 1$, $\tilde P(t)$ tends to a plane orthogonal to Face 6, dividing the polyhedron $\tilde P(1)$ into two equal parts. By continuity, for some subinterval of $(1,2)$, $\tilde P(t)$ intersects the edge shared by Faces 7 and 8. Another continuity argument guarantees the existence of a $t=t_0\in(1,2)$ so that $\tilde P(t_0)$ intersects this edge orthogonally. This implies that $\tilde P(t_0)$ intersects both Faces 7 and 8 orthogonally, and disjoint from Face 5. See Figure \ref{fig: schematic} for a picture of $\tilde P(t_0)\cap\tilde Q(t_0)$.
\begin{figure}[htp]
\centering
\includegraphics{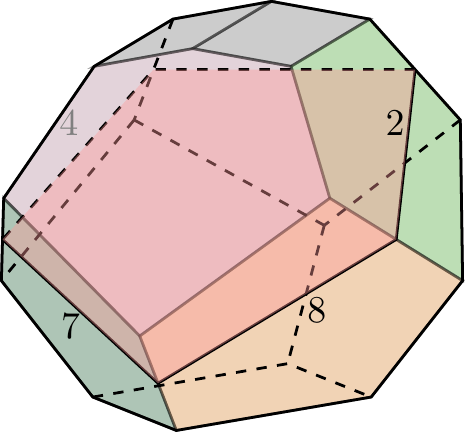}
\caption{A piece of the exotic plane $P$ inside the fundamental domain}
\label{fig: schematic}
\end{figure}

Clearly, as $\tilde P(t_0)$ is orthogonal to Faces 2, 4, 7 and 8, the half plane that projects to a cylinder in the bottom end in $N(t_0)$ further projects down to a orbifold surface in $\mathcal{O}(t_0)$, with one geodesic boundary component, two cone points of order 2 and one cone point of order 3. The other half behaves exactly as that in the quasifuchsian example (recall that a copy of $\core(N(t_0))$ is embedded in $\core(\mathcal{O}(t_0))$, see Figure~\ref{fig: poly bending}), so it is indeed an exotic plane in $\mathcal{O}(t_0)$. Following Theorem~\ref{cor: qf_geometry}, the closure of this plane is not locally connected.
\end{proof}
Theorem~\ref{thm: main} is then a direct consequence of Theorem~\ref{thm: exotic}, after taking a finite manifold cover of $\mathcal{O}(t_0)$.

\section{Computations}\label{sec: computation}
In this section, we calculate the explicit value of $t_0$ predicted in the previous section. We refer to the calculations in \cite[\S6]{double_lunchbox} freely.

It is convenient to work with the hyperboloid model of $\mathbb{H}^3$. Let $E^{3,1}$ be $\mathbb{R}^4$ equipped with the indefinite inner product $\langle x,y\rangle=-x_0y_0+x_1y_1+x_2y_2+x_3y_3$, and set
$$V_+:=\{x\in E^{3,1}:\langle x,x\rangle=-1, x_0>0\}.$$
Together with the metric induced from the inner product, $V_+$ is isometric to $\mathbb{H}^3$. We can uniquely determine a geodesic plane in $\mathbb{H}^3$ using its unit normal in the hyperboloid model: given $e\in E^{3,1}$ with $\langle e,e\rangle=1$, $V_+\cap (\mathbb{R}e)^\perp$ is a geodesic plane. For correspondence between this model and the upper half space/unit ball model, especially for geodesic planes, we refer to \cite{double_lunchbox, mythesis}.

Set $u=\sqrt{(t+1)/2}$. One choice of normals for Faces 2 and 4 of the polyhedron $\tilde Q(t)$ may be
$$\left(\frac{u-\sqrt{4u^2+4v^2-3-4u^2v^2}}{2u^2-2},\pm v,0,\frac{-1+u\sqrt{4u^2+4v^2-3-4u^2v^2}}{2u^2-2}\right)$$
where $v=\frac{3u+\sqrt{(u^2+2)(16u^2-3)}}{8u^2-2}$. Correspondingly, the normals for Faces 7 and 8 are
$$\left(\frac{u^2\sqrt{3}-\sqrt{u^2+2}}{2u^2-2},\mp\frac{\sqrt{3}}{2},\frac{\sqrt{3}}2,\frac{u\sqrt{u^2+2}-u\sqrt{3}}{2u^2-2}\right).$$
Suppose the unit normal to the plane $\tilde P$ is $(x_0,x_1,x_2,x_3)$. Then this vector is orthogonal to the normals listed above. Therefore $x_1=0$ and
$$x_2=\frac{\sqrt{u^2+2}-u\sqrt{3}\sqrt{4u^2+4v^2-3-4u^2v^2}}{\sqrt{3}(-1+u\sqrt{4u^2+4v^2-3-4u^2v^2})}x_0,\quad x_3=\frac{u-\sqrt{4u^2+4v^2-3-4u^2v^2}}{-1+u\sqrt{4u^2+4v^2-3-4u^2v^2}}x_0.$$
On the other hand, the sequence of planes $\tilde{\eta}^n\cdot\tilde P$ converges to $\tilde P'$, where $\tilde P'$ has unit normal
$$\left(\frac1{\sqrt{u^2-1}},0,0,-\frac{u}{\sqrt{u^2-1}}\right).$$
This can be calculated using the formula in \cite[\S6.2]{double_lunchbox} for the hyperbolic element $\tilde\eta$, and the fact that the circle $C'$ corresponding to $\tilde P'$ passes through the fixed points of $\tilde\eta$ and is symmetric across the imaginary axis. Since $\tilde P$ is tangent to $\tilde P'$ at infinity, the inner product of their unit norms is $1$ (or $-1$, but we can always change the orientation of $\tilde P$), so
$$-x_0\frac{1}{\sqrt{u^2-1}}-x_3\frac{u}{\sqrt{u^2-1}}=1.$$
Hence we have
\begin{equation*}
\begin{split}
x_0&=\frac{1-u\sqrt{4u^2+4v^2-3-4u^2v^2}}{\sqrt{u^2-1}},\\
x_2&=\frac{-\sqrt{u^2+2}+u\sqrt{3}\sqrt{4u^2+4v^2-3-4u^2v^2}}{\sqrt{3}\sqrt{u^2-1}},\\
x_3&=\frac{\sqrt{4u^2+4v^2-3-4u^2v^2}-u}{\sqrt{u^2-1}}.
\end{split}
\end{equation*}
Since $-x_0^2+x_2^2+x_3^2=1$, we have
$$4u^2+4v^2-3-4u^2v^2-\frac{2\sqrt{3}}{3}u\sqrt{u^2+2}\sqrt{4u^2+4v^2-3-4u^2v^2}+\frac{u^2+2}{3}=0,$$
and thus
\begin{equation*}
\begin{split}
9(4u^2+4v^2-3-4u^2v^2)^2+(u^2+2)^2+6(4u^2+4v^2&-3-4u^2v^2)(u^2+2)\\
&-12u^2(u^2+2)(4u^2+4v^2-3-4u^2v^2)=0.
\end{split}
\end{equation*}
Plugging in the expression of $v$ in terms of $u$, the left hand side gives
 $$\frac{u^2-1}{(4u^2-1)^4}\left(f(u)+g(u)\sqrt{16u^4+29u^2-6}\right)$$
 where
 \begin{equation*}
 \begin{split}
 f(u)&=-625 + 11153 u^2 - 53284 u^4 + 65632 u^6 - 38720 u^8 + 22144 u^{10} - 
 9216 u^{12},\\
 g(u)&=900 u - 7092 u^3 + 9072 u^5 - 4032 u^7 + 1152 u^9.
 \end{split}
 \end{equation*}
 Thus we have
 \begin{equation*}
 \begin{split}
 0&=f(u)^2-g(u)^2(16u^4-29u^2-6)\\
 &=(4u^2-1)^4(-625 + 944 u^2 - 976 u^4 + 576 u^6)\\
 &\qquad\qquad\qquad\qquad(-625 + 3586 u^2 - 6585 u^4 + 3112 u^6 - 1360 u^8 + 576 u^{10})\\
 &=\frac14 (1 + 2 t)^4 (-325 + 200 t - 28 t^2 + 72 t^3) (-625 - 2330 t - 3237 t^2 + 916 t^3 + 20 t^4 + 72 t^5)
 \end{split}
 \end{equation*}
 Set $h(t):=-625 - 2330 t - 3237 t^2 + 916 t^3 + 20 t^4 + 72 t^5$. Then $h''(t)=-6474 + 5496 t + 240 t^2 + 1440 t^3>0$ when $t\in(1,2)$. So $h(t)$ attains maximum at $t=1$ or $t=2$. But $h(1)=-5184<0$ and $h(2)=-8281<0$, so $h(t)<0$ for $t\in(1,2)$. To find $t\in(1,2)$ satisfying the equation above, we must thus solve
 $$-325 + 200 t - 28 t^2 + 72 t^3=0.$$
 This polynomial has a unique real root
 $$t=t_0=\frac1{54}\left(7 - \left(\frac{25515 \sqrt{773}-654761}{2}\right)^{1/3} + \left(\frac{25515 \sqrt{773}+654761}{2}\right)^{1/3}\right)\approx1.202,$$
in $(1,2)$ as desired.

\appendix
\section{Appendix: Countably many exotic planes}
In this appendix, we give a proof of Theorem~\ref{thm: countable}; that is, there are at most countably many exotic planes in a convex cocompact hyperbolic 3-manifold $M=\Gamma\backslash\mathbb{H}^3$ with incompressible boundary.

We start by recalling the following theorem in \cite{MMO2}:
\begin{thm}[\cite{MMO2}]
If $M$ is convex cocompact with incompressible boundary, then the fundamental group of any plane $P$ with $P^*$ nonempty and closed in $M^*$ is nontrivial.
\end{thm}
In particular, either $P^*$ is a cylinder, or $P^*$ is a nonelementary surface. For the latter case, the fundamental group of $P^*$ contains a free group on two generators. Following the same arguments of the proof of \cite[Cor.~2.3]{MMO2}, we conclude that there are only countably many nonelementary $P^*$.

For the former case, $P^*$ contains a closed geodesic $\gamma$ of $M$. Let $\tilde\gamma$ be any lift, with end points $p,q$ on the sphere at infinity. Let $C$ be the boundary circle of $P$ passing through $p,q$. We have
\begin{lm}
Each component of $C-\{p,q\}$ is contained in the closure of a connected component of the domain of discontinuity $\Omega$.
\end{lm}
\begin{proof}
We follow many of the arguments in \cite[\S2]{MMO2}. The plane $P$ is the image of $\hull(C)\cong\mathbb{H}^2$ under an isometric immersion, which descends to a map $f:S\to M$ where $S=\Stab_{\Gamma}(C)\backslash\hull(C)$ is a hyperbolic cylinder. Let $S^*=f^{-1}(M^*)$. Then $S^*$ is a convex subsurface of $S$, also a topological cylinder, and the restriction $f:S^*\to P^*\subset M^*$ is proper.

Since $\core(M)$ is homeomorphic to $\overline{M}$, $M^*$ deformation retracts to a compact submanifold $K\subset M^*$. One can also arrange that $K$ is transverse to $f$, and so $S_0:=f^{-1}(K)\subset S^*$ is a compact, smoothly bounded region in $S^*$, although not necessarily connected.

As in the proof of Theorem 2.1 in \cite{MMO2}, after changing $f$ by a compact deformation, we may assume that the inclusion of $S_0$ into $S^*$ is injective on $\pi_1$. So the components of $S_0$ are either cylinders, or disks; see Figure~\ref{fig: cylinder}.
\begin{figure}[htp]
\centering
\includegraphics[width=0.7\textwidth]{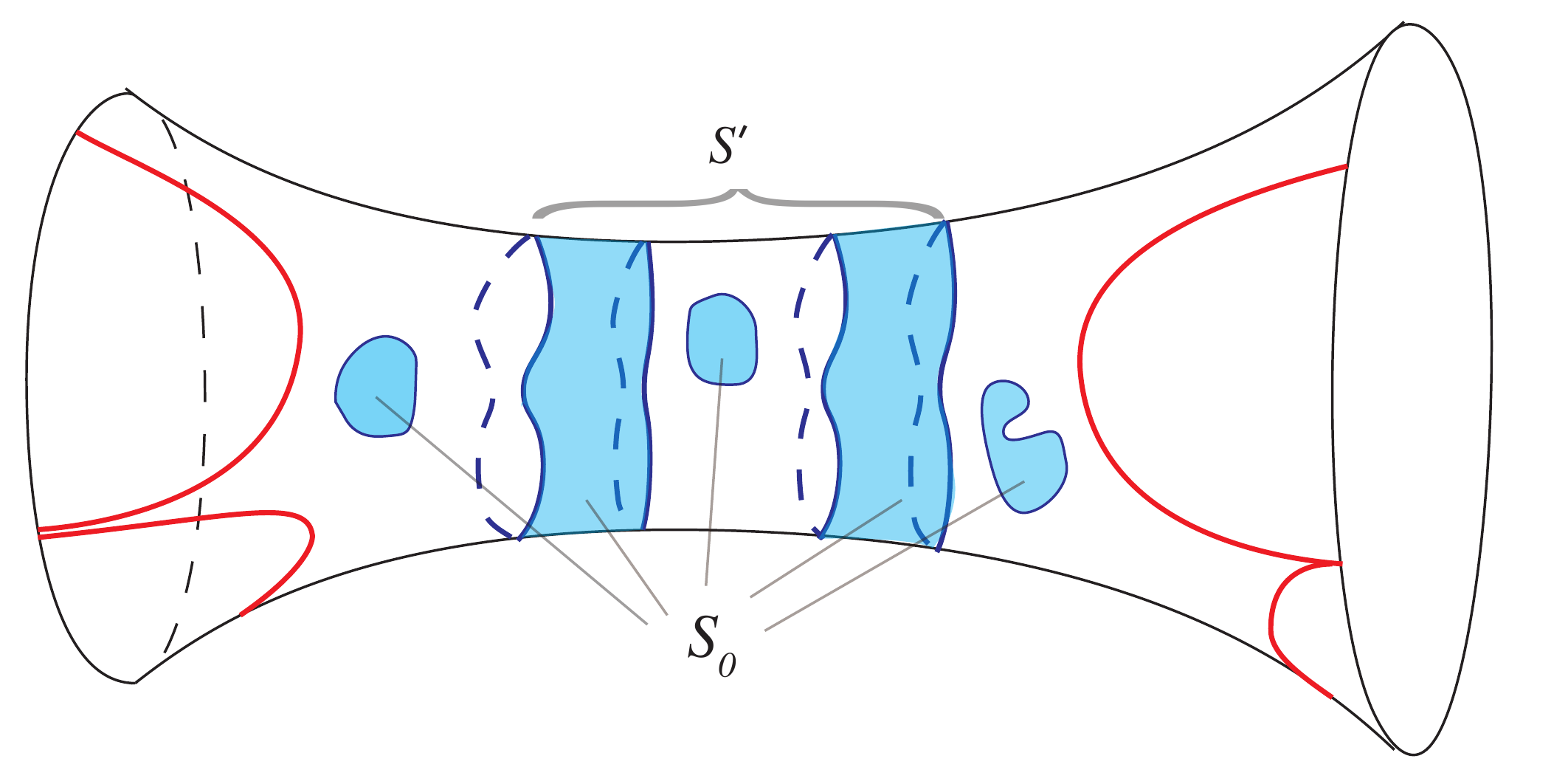}
\caption{The hyperbolic cylinder $S$, the convex subsurface $S^*$ bounded by red, and $S_0$}
\label{fig: cylinder}
\end{figure}

Boundaries of the cylinder components of $S_0$ are essential curves of $S^*$, all homotopic; let $S'$ be the region of $S^*$ bounded by the two of these boundary curves farthest into either end of $S$. In particular $S'$ is also a cylinder, and include all cylinder components of $S_0$. For each component of $S_0$ not contained in $S'$, since it is a disk with boundary in $\partial K$, we may homotope it rel boundary into $\partial K$, as $K\cong \overline{M}$ has incompressible boundary.

Therefore after changing $f$ further by a compact deformation, we may divide $S$ into $S'$ and two ends $S_1,S_2$ so that $f(S')$ is a cylinder with both boundary components on $\partial K$, and $f(S_1)$ and $f(S_2)$ are contained completely in $M-K$.

Let $E_1,E_2$ be the two components of $M-K$ containing $f(S_1)$ and $f(S_2)$. Note that the closed geodesic $\gamma$ contained in $P$ is homotopic to an essential curve $\gamma_i$ on $\partial E_i$ for $i=1,2$. Moreover, $E_i$ differs from an actual end of $M$ by a compact set, so any lift of $E_i$ is bounded at infinity by a connected component of $\Omega$. If we take the lift corresponding to $\tilde\gamma$,  let $\Omega_i$ be the connected component bounding the lift of $E_i$ at infinity. Since $p,q$ are the end points of $\tilde{\gamma_i}$, we must have $p,q\in\partial \Omega_i$. Finally, since $f$ differs from an isometric immersion by compact deformation, the lift of $f(S_i)$ is a half plane bounded by $\tilde{\gamma_i}$ and a component of $C-\{p,q\}$ at infinity; this half plane is totally geodesic except in a band of bounded width near $\tilde{\gamma_i}$. This lift is contained completely in the lift of $E_i$, so the corresponding component of $C-\{p,q\}$ is contained in $\overline{\Omega_i}$.
\end{proof}

Given a closed geodesic $\gamma\subset M$, there is an $S^1$-family of planes passing through $\gamma$. Fix any continuous parametrization $P_t$ of this family by $t\in \mathbb{R}$ invariant under translation by an integer. Consider the set
$$L:=\{t\in\mathbb{R}:P_t^*\text{ is a properly immersed cylinder in }M^*\}.$$
We claim
\begin{prop}
Suppose a sequence $\{t_i\}\subset L$ satisfies $t_i>t$ and $t_i\to t$ for some $t\in L$. Then there exists $\epsilon>0$ so that $[t,t+\epsilon]\subset L$ and $P_s$ is not exotic for any $s\in (t,t+\epsilon)$.
\end{prop}
\begin{proof}
Choose any lift $\tilde\gamma$ of the closed geodesic $\gamma$, and a boundary circle $C$ of $P_t$ passing through the end points $p,q$ of $\Gamma$. By the previous lemma, each of the two components of $C-\{p,q\}$ is contained in the closure of a connected component of the domain of discontinuity $\Omega$, say $\Omega_1$ and $\Omega_2$.

For each $i$, a boundary circle $C_i$ of $P_{t_i}$ also passes through $p,q$. Since $C_i\to C$ as $i\to\infty$, when $i$ is large enough, the two components of $C_i-\{p,q\}$ intersect $\overline{\Omega_1}$ and $\overline{\Omega_2}$ respectively. As $P_{t_i}^*$ is also a cylinder, those two components must be contained in $\overline{\Omega_1}$ and $\overline{\Omega_2}$ respectively.

Finally, since $M$ has incompressible boundary, $\Omega_1$ and $\Omega_2$ are both simply connected. In particular, for any $s\in (t,t_i)$, a boundary circle $C_s$ of $P_s$ is sandwiched between $C$ and $C_i$, and $C_s$ intersects the limit set at exactly $p$ and $q$. It is easy to see $P_s$ itself is closed in $M$.
\end{proof}
Similarly, a sequence $\{t_i\}\subset L$ tending to $t\in L$ from below gives a corresponding interval with $t$ as the right end point. We have
\begin{cor}
Every point in the set $\{t\in L:P_t\text{ is exotic}\}$ is isolated, and therefore this set is countable.
\end{cor}
Since $M$ contains countably many closed geodesics, we conclude that there are only countably many exotic cylinders. This, together with the fact that there are only countably many nonelementary $P^*$, gives Theorem~\ref{thm: countable}.

\bibliographystyle{math}
\bibliography{biblio}

\begin{thebibliography}{{Zha}2}

\bibitem[BO]{acy_geom_finite}
Y.~{Benoist} and H.~{Oh}.
\newblock {{Geodesic planes in geometrically finite acylindrical 3-manifolds}}.
\newblock {\em Ergodic Theory Dynam. Systems} {\bf 42}(2022), 514--553.

\bibitem[Dal]{horocycle}
F.~Dal'bo.
\newblock {Topologie du feuilletage fortement stable}.
\newblock {\em Ann. Inst. Fourier} {\bf 50}(2000), 981--993.

\bibitem[Fri]{poly2}
R.~Frigerio.
\newblock {An infinite family of hyperbolic graph complements in {$S^3$}}.
\newblock {\em J. Knot Theory Ramifications} {\bf 14}(2005), 479--496.

\bibitem[Gas]{gaster}
J.~Gaster.
\newblock {A family of non-injective skinning maps with critical points}.
\newblock {\em Trans. Amer. Math. Soc.} {\bf 368}(2016), 1911--1940.

\bibitem[{Geo}]{geomview}
{Geometry Center}.
\newblock {Geomview 1.9.5}.
\newblock Available at \url{http://www.geomview.org/}, March 2014.

\bibitem[McM]{lim}
C.~T. McMullen.
\newblock {lim}.
\newblock Available at
  \url{http://people.math.harvard.edu/~ctm/programs/home/prog/lim/src/lim.tar}.

\bibitem[MMO1]{acy_horo}
C.~T. McMullen, A.~Mohammadi, and H.~Oh.
\newblock {Horocycles in hyperbolic 3-manifolds}.
\newblock {\em Geom. Funct. Anal.} {\bf 26}(2016), 961--973.

\bibitem[MMO2]{MMO1}
C.~T. McMullen, A.~Mohammadi, and H.~Oh.
\newblock {Geodesic planes in hyperbolic 3-manifolds}.
\newblock {\em Invent. Math.} {\bf 209}(2017), 425--461.

\bibitem[MMO3]{MMO2}
C.~T. {McMullen}, A.~{Mohammadi}, and H.~{Oh}.
\newblock {{Geodesic planes in the convex core of an acylindrical 3-manifold}}.
\newblock {\em To appear, Duke Math. J.}

\bibitem[PZ]{poly1}
L.~Paoluzzi and B.~Zimmermann.
\newblock {On a class of hyperbolic 3-manifolds and groups with one defining
  relation}.
\newblock {\em Geom. Dedicata} {\bf 60}(1996), 113--123.

\bibitem[Rat]{ratner}
M.~Ratner.
\newblock {Raghunathan's topological conjecture and distributions of unipotent
  flows}.
\newblock {\em Duke Math. J.} {\bf 63}(1991), 235--280.

\bibitem[Roe]{roe}
R.~K.~W. Roeder.
\newblock {Constructing hyperbolic polyhedra using Newton's method}.
\newblock {\em Exp. Math.} {\bf 16}(2007), 463--492.

\bibitem[Sha]{shah}
N.~Shah.
\newblock {Closures of totally geodesic immersions in manifolds of constant
  negative curvature}.
\newblock In {\em Group Theory from a Geometrical Viewpoint (Trieste, 1990)},
  pages 718--732. World Scientific, 1991.

\bibitem[Thu1]{hyperbolization1}
W.~P. Thurston.
\newblock {Hyperbolic structures on 3-manifolds I: Deformation of acylindrical
  manifolds}.
\newblock {\em Ann. of Math.} {\bf 124}(1986), 203--246.

\bibitem[Thu2]{thurston_book}
W.~P. Thurston.
\newblock {\em Three-Dimensional Geometry and Topology}, volume~1.
\newblock Princeton University Press, 1997.

\bibitem[{Zha}1]{mythesis}
Y.~{Zhang}.
\newblock {\em Geodesic planes in hyperbolic 3-manifolds}.
\newblock PhD thesis, Harvard University, 2021.

\bibitem[{Zha}2]{double_lunchbox}
Y.~{Zhang}.
\newblock {Construction of acylindrical hyperbolic 3-manifolds with
  quasifuchsian boundary}.
\newblock {\em To appear, Exp. Math.}

\end{thebibliography}
\end{document}